\documentclass{amsart}
\usepackage[utf8]{inputenc}
\usepackage{amsmath}
\usepackage{amsfonts}
\usepackage{amsthm}
\usepackage{amssymb}
\usepackage{stmaryrd}
\usepackage{mathtools}
\usepackage{tikz}
\usepackage{tikz-cd}
\usepackage{float}
\usepackage{import}
\usepackage{longtable}
\usepackage{makecell}
\usepackage{geometry}
\restylefloat{table}

\newtheorem{theorem}{Theorem}[section]
\newtheorem{corollary}{Corollary}[theorem]

\newtheorem{lemma}[theorem]{Lemma}
\newtheorem{proposition}[theorem]{Proposition}

\makeatletter
\newcommand*{\da@rightarrow}{\mathchar"0\hexnumber@\symAMSa 4B }
\newcommand*{\da@leftarrow}{\mathchar"0\hexnumber@\symAMSa 4C }
\newcommand*{\xdashrightarrow}[2][]{%
  \mathrel{%
    \mathpalette{\da@xarrow{#1}{#2}{}\da@rightarrow{\,}{}}{}%
  }%
}
\newcommand{\xdashleftarrow}[2][]{%
  \mathrel{%
    \mathpalette{\da@xarrow{#1}{#2}\da@leftarrow{}{}{\,}}{}%
  }%
}
\newcommand*{\da@xarrow}[7]{%
  \sbox0{$\ifx#7\scriptstyle\scriptscriptstyle\else\scriptstyle\fi#5#1#6\m@th$}%
  \sbox2{$\ifx#7\scriptstyle\scriptscriptstyle\else\scriptstyle\fi#5#2#6\m@th$}%
  \sbox4{$#7\dabar@\m@th$}%
  \dimen@=\wd0 %
  \ifdim\wd2 >\dimen@
    \dimen@=\wd2 %
  \fi
  \count@=2 %
  \def\da@bars{\dabar@\dabar@}%
  \@whiledim\count@\wd4<\dimen@\do{%
    \advance\count@\@ne
    \expandafter\def\expandafter\da@bars\expandafter{%
      \da@bars
      \dabar@ 
    }%
  }%
  \mathrel{#3}%
  \mathrel{%
    \mathop{\da@bars}\limits
    \ifx\\#1\\%
    \else
      _{\copy0}%
    \fi
    \ifx\\#2\\%
    \else
      ^{\copy2}%
    \fi
  }%
  \mathrel{#4}%
}
\makeatother

\usepackage{hyperref}
\usepackage[backend=biber,style=alphabetic,url=false,doi=false,isbn=false,sorting=nyt,maxbibnames=9,giveninits=true]{biblatex}
\AtEveryBibitem{\clearfield{month}}
\AtEveryBibitem{%
  \clearlist{language}%
}

\theoremstyle{definition}
\newtheorem{definition}[theorem]{Definition}
\newtheorem{remark}[theorem]{Remark}

\title{GIT of complete intersections and log canonical thresholds}
\address{School of Mathematics and Statistics, University of Glasgow, University Place, Glasgow, G12 8QQ, United Kingdom}
\email{theodorosstylianos.papazachariou@glasgow.ac.uk}
\author{Theodoros Stylianos Papazachariou}
\date{December 2022}

\begin{document}
\maketitle

\begin{abstract}
We establish a link between the GIT stability of complete intersections of same degree hypersurfaces in the same ambient projective space, which can be parametrised as tuples in a Grassmanian scheme, and the log canonical thresholds of hypersurfaces contained in these tuples. We also establish a  similar link in the case of Variations of GIT quotients, where we consider tuples of complete intersections and distinct hyperplanes.
\end{abstract}

\section{Introduction}\label{intro}

Geometric Invariant Theory (GIT), is a powerful theory  developed by Mumford in the 60s, used to model moduli spaces for algebraic varieties. In recent years, GIT has found particular use in compactifying moduli spaces of Fano varieties, as it has been used to explicitly compactify moduli spaces that arise from K-stability \cite{odaka_spotti_sun_2016, liu-xu-threefolds, liu2020kstability, me}. In finding such explicit compactifications a necessary step is to describe stable, polystable and semistable points of the corresponding GIT quotient.

Computational methods to study GIT quotients are not particularly novel, with a number of different constructions present \cite{laza-cubics, gallardo_martinez-garcia_2018, Gallardo_2020, me}, but in many cases detecting stable and semistable points in the GIT quotient remains a challenge. Recent results \cite{zanardini2021stability, zanardini-hattori} have related stability to the study of log canonical thresholds. In this note we expand on the computational setting presented in \cite{me}, and the techniques of \cite{zanardini2021stability}, to study a link between the GIT stability of complete intersections of same degree hypersurfaces in the same ambient projective space, which can be parametrised as tuples in a Grassmanian scheme, and the log canonical thresholds of hypersurfaces in these tuples. We also extend this analysis to Variations of GIT quotients, where we consider tuples of complete intersections and distinct hyperplanes, and obtain a link between VGIT stability, different polarisations and log canonical thresholds.

In more detail, working over $\mathbb{C}$, we consider complete intersections $S = \{f_1=0\}\cap\dots\cap \{f_k=0\}$, where each $f_i$ is a degree $d$ polynomial in $\mathbb{P}^n$. The complete intersection $S$ can be thought of as the base locus of the tuple $\mathcal{T} = \{\sum_{i=1}^kz_if_i|(z_1\colon\dots\colon z_k)\in \mathbb{P}^{k-1}\}$, which is the linear system of the defining polynomials. Let $\mathcal{R}_{n,d,k}$ be the parametrising space of all such $k$-tuples.

For a hypersurface $F = \{f=0\}$ the log canonical threshold is the number 
$$\mathrm{lct}(\mathbb{P}^n, F ) = \sup\{t \in \mathbb{Q}  | (\mathbb{P}^n, tF )\text{ is log canonical at p}\}.$$
The main aim of this paper is to establish a link between the stability of tuples $\mathcal{T}$, and the log canonical threshold $\mathrm{lct}(\mathbb{P}^n, F )$ of hypersurfaces $F\in \mathcal{T}$.

Given this setting, our first main result is the following.

\begin{theorem}[see Corollary \ref{if_ss_lct_condition}]
If $\mathcal{T}$ is (semi-)stable, then
for any hypersurface $F =\{f=0\} \in\mathcal{T}$ and any base point $p$ of $\mathcal{T}$ 
$$\frac{n+1}{kd}< \mathrm{lct}_p(\mathbb{P}^n, F) \text{ (respectively } \leq).$$

\end{theorem}

We also establish a method to detect the stability of the tuple $\mathcal{T}$ using the log canonical threshold of each hypersurface in $\mathcal{T}$.

\begin{theorem}[see Corollary \ref{zanardini corollary to go in intro}]
    If a tuple $\mathcal{T}\in \mathcal{R}_{n,d,k}$ is such that $\mathrm{lct}(\mathbb{P}^n,F)\geq \frac{n+1}{d}$ (respectively, $>\frac{n+1}{d}$)
for any hypersurface $F = \{f=0\}$ in $\mathcal{T}$, then $\mathcal{T}$ is semistable (respectively, stable).
\end{theorem}




\begin{remark}
    While preparing this note, we were made aware that the authors of \cite{zanardini-hattori} had achieved similar results to the above Theorems independently. In their analysis, they study not just complete intersections but also linear systems and achieve a stronger if and only if result on the GIT stability and log canonical thresholds \cite[Theorem 1.1]{zanardini-hattori}. 

    They also further demonstrate the applications of their results, by providing explicit examples. In particular, they are able to recover results by Miranda on pencils of plane cubics \cite[\S 4.1]{miranda_1980}, and Wall on nets of conics \cite[\S 5]{wall_1983} by using their main Theorems. We would like to thank the authors of \cite{zanardini-hattori} for both pointing this out, and for very helpful suggestions in improving this manuscript.

    We should note that in \cite{zanardini-hattori}, the case of VGIT is not considered, and as such the results of Section \ref{VGIT section} are new. 
\end{remark}

Considering tuples $(S,H_1,\dots,H_m)$ where $S$ is as before and the $H_i$ are $m$ distinct hyperplanes, the GIT stability conditions change, and they depend on the choice of linearisation. By \cite{dolgachev_1994} and \cite{thaddeus_1996} there exist a finite number of stability conditions. In this Variations of GIT setting, we show the following.

\begin{theorem}[see Corollary \ref{vgit_corollary_zan}]
    If a tuple $(\mathcal{T},H_1\dots,H_m)\in \mathcal{R}_{m}$ is such that $\mathrm{lct}(\mathbb{P}^n,F)\geq \frac{k(n+1)}{kd-n\sum_{i=1}^mt_i}$ (respectively, $>\frac{k(n+1)}{kd-n\sum_{i=1}^mt_i}$)
for any hypersurface $f$ in $\mathcal{T}$, then $(\mathcal{T},H_1\dots,H_m)\in \mathcal{R}_{m}$ is $\vec{t}-$semistable (respectively, $\vec{t}-$stable).
\end{theorem}

The organisation of the paper is as follows. In Section \ref{VGIT_prelims} we make a quick summary of the relevant GIT setting. In Section \ref{git_lc_sect} we demonstrate how the Hilbert-Mumford numerical criterion allows us to establish a link between stability and the log canonical threshold. In Section \ref{VGIT section} we establish a similar link between VGIT stability and the log canonical threshold.

\renewcommand{\abstractname}{Acknowledgements}
\begin{abstract}
I would like to thank my PhD supervisor, Jesus Martinez-Garcia for his useful comments and suggestions in improving this draft. 
This paper is part of my PhD thesis at the University of Essex, funded by a scholarship of the Department of Mathematical Sciences. Parts of this paper were finished during my postdoctoral position at the University of Glasgow, funded by Ruadha´ıDervan’s Royal Society University Research Fellowship. I would also like to thank him for helpful suggestions. I would also like to thank Masafumi Hattori and Aline Zanardini for very helpful suggestions in improving this manuscript. I would also like to thank my PhD defence examiners, Chenyang Xu and Gerald Williams, for many useful comments and suggestions in improving this draft. 
\end{abstract}

\section{Preliminaries in GIT}\label{VGIT_prelims}
Throughout this Section, we will work over an algebraically closed field $k$. 
Let $G\coloneqq \operatorname{SL}(n+1)$. Consider a variety $S$ which is the complete intersection of $k$ hypersurfaces of degree $d$ in $\mathbb{P}^n$, i.e, $S = \{f_1 = f_2 = \dots = f_k=0\}$, where each $f_i(x) = \sum f_{I_i}x^{I_i} $ with $I_i = \{d_{i,0}, \dots, d_{i,n} \}$, $\sum_{j=0}^n d_{i,j}=d$ for all $i$. Here, $x^{I_i} = x_0^{d_{i,0}} x_1^{d_{i,1}}\dots x_n^{d_{i,n}}$. Also consider $H_1,\dots, H_m$, $m$ distinct hyperplanes with defining polynomials $h_i(x) = \sum h_{i,j}x_j$.  Let $\Xi_d$ be the set of monomials of degree $d$ in variables $x_0, \dots, x_n$, written in the vector notation $I_i =(d_{i,0},\dots, d_{i,n})$. As in Gallardo --Martinez-Garcia \cite{gallardo_martinez-garcia_2018, zhang_2018} and \cite{me} we define the associated set of monomials $$\operatorname{Supp}(f_i)= \{x^{I_i} \in \Xi_d|f_{I_i} \neq 0\},\quad \operatorname{Supp}(h_i)= \{x_j \in \Xi_1|h_{i,j} \neq 0\}.$$

Let $V \vcentcolon= \operatorname{H}^0(\mathbb{P}^n, \mathcal{O}_{\mathbb{P}^n}(1))$ and $W \vcentcolon= \operatorname{H}^0(\mathbb{P}^n, \mathcal{O}_{\mathbb{P}^n}(d))  \simeq \operatorname{Sym}^d V $ be the vector space of degree $d$ forms.  For an embedded variety $S = (f_1,\dots, f_k) \subseteq \mathbb{P}^n $ we associate its Hilbert point

$$[S] = [f_1 \wedge \dots \wedge f_k] \in \operatorname{Gr}(k, W) \subset \mathbb{P} \bigwedge^k W.$$
Note that $\operatorname{Gr}(k, W)$ is embedded in $\mathbb{P} \bigwedge^k W$ via the Pl{\"u}cker embedding $(w_1,\dots, w_r)\rightarrow [w_1\wedge\dots\wedge w_r]$, where the $w_r$ are the basis vectors of $W$.  We denote by $[\overline{S}] \vcentcolon= f_1 \wedge \dots \wedge f_k$ some lift in $\bigwedge^k W$. We will consider the natural $G$ action, given by $A\cdot f(x) = f(Ax)$ for $A\in G$.

For simplicity, we will denote $\operatorname{Gr}(k,W)$ by $\mathcal{R}_{n,d,k}$, and we let $\mathcal{R}_m \vcentcolon= \mathcal{R}_{n,d,k}\times \Big(\mathcal{R}_{n,1,1}\Big)^m$ be the parameter scheme of tuples $(f_1,\dots f_k, h_1, \dots, h_m)$, under the identification
$(f_1,\dots f_k)=c(f_1,\dots f_k)$ and $h_i=c_ih_i$ for $c,c_i\in \mathbb{G}_m$, where:
$$\mathcal{R}_{m} \cong \operatorname{Gr}\Big(k,{\binom{n+d}{d}}\Big) \times \Big(\mathbb{P}(H^0(\mathbb{P}^n, \mathcal{O}_{\mathbb{P}^n}(1)))\Big)^m \hookrightarrow \mathbb{P}\bigwedge^k W \times (\mathbb{P}^n)^m.$$
In the case where $m = 1$, we will just write $\mathcal{R} = \mathcal{R}_1$. There is a natural $G$ action on $V$ and $W$ given by the action of $G$ on $\mathbb{P}^n$. This action induces an action of $G$ on $\mathcal{R}_{n,d,k}$ via the natural maps, and by the inclusion map to the Pl{\"u}cker embedding $\mathbb{P}\bigwedge^k W$. By extension, we also obtain an induced action of $G$ to $\mathcal{R}$. We aim to study the GIT quotients $\mathcal{R}_m\sslash G$.

Let $\mathcal{C}\coloneqq\mathbb{P}(W)$. We will begin our analysis by first studying the GIT quotients $\mathcal{C}\sslash G$. $\mathcal{C}$ parametrises hypersurfaces $X = \{f=0\}$ of degree $d$ in $\mathbb{P}^n$, where $f = \sum f_Ix^I$ is a polynomial of degree $d$.  We must study the Hilbert-Mumford numerical criterion in order to study this quotient. In order to do so, we fix a maximal torus $T\cong (\mathbb{G}_m)^n \subset G$ which in turn induces lattices of characters $M = \operatorname{Hom}_{\mathbb{Z}}(T, \mathbb{G}_m) \cong \mathbb{Z}^{n+1}$ and one-parameter subgroups $N = \operatorname{Hom}_{\mathbb{Z}}( \mathbb{G}_m, T) \cong \mathbb{Z}^{n+1}$ with natural pairing 
$$\langle - , -\rangle \colon M\times N \rightarrow \operatorname{Hom}_{\mathbb{Z}}( \mathbb{G}_m, \mathbb{G}_m) \cong \mathbb{Z}$$
given by the composition $(\chi,\lambda) \mapsto \chi \circ \lambda$. We also choose projective coordinates $(x_0 \colon \dots \colon x_n)$ such that the maximal torus $T$ is diagonal in $G$. Given a one-parameter subgroup $\lambda \colon \mathbb{G}_m \rightarrow T \subset SL(n+1)$ we say $\lambda$ is \textit{normalised} if 
$$\lambda(s) = \operatorname{Diag} (s^{\mu_0}, \dots, s^{\mu_n})$$ where $\mu_0 \geq \dots \geq \mu_n$ with $\sum \mu_i = 0$ (implying $\mu_0 >0, \mu_n <0$ if $\lambda$ is not trivial).

From \cite[Lemma 3.8]{me} with $m=1$, we can choose an ample $G$-linearisation $\mathcal{L} = \mathcal{O}_{\mathcal{C}}(1)$. Hence, the set of characters, with respect to this $G$-linearisation corresponds to degree $d$ polynomials $f$, as sections $s\in H^0(\mathcal{C}, \mathcal{O}_{\mathcal{C}}(1))$ are polynomials of degree $d$. In particular, for a monomial $x^I$ of degree $d$ and a normalised one-parameter subgroup $\lambda$ as above, the natural pairing is given by $\langle x^I, \lambda \rangle = \sum_{i=0}^n d_i\mu_i$. Note, that in many cases we will abuse the notation and write $\langle I, \lambda \rangle$ instead of $\langle x^I, \lambda \rangle$. Then, the $G$-action induced by $\lambda$ on a monomial $x^I$ is given by $\lambda(s)\cdot x^I = s^{\langle I, \lambda \rangle}x^I$. Naturally, the $G$-action induced by $\lambda$ on the polynomial $f$ is given by 
$$\lambda(s)\cdot f = \sum_{I\in \operatorname{Supp}(f)} s^{\langle I, \lambda \rangle}f_I x^I.$$
In addition, notice that the action of $\lambda$ on a fiber is equivalent to the action of $\lambda$ on the polynomial $f$, and by the above discussion, we have that $\operatorname{weight}(f,\lambda) = \min_{x^I\in \operatorname{Supp}(f)}\{\langle I, \lambda \rangle\}$. Thus the Hilbert-Mumford function reads
$$\mu(f,\lambda) = - \min_{x^I\in \operatorname{Supp}(f)}\{\langle I, \lambda \rangle\},$$
as in \cite[\S 3.1]{me} and the Hilbert-Mumford numerical criterion is:
\begin{lemma}[{\cite[Lemma 3.5]{me}}]\label{HM criterion for hypersurfaces}
With respect to a maximal torus $T$:
\begin{enumerate}
    \item $X = \{f=0\}$ is \emph{semi-stable} if and only if $\mu(f,\lambda) \geq 0$, i.e. if $\min_{x^I\in \operatorname{Supp}(f)}\{\langle I, \lambda \rangle\}\leq 0$, for all non-trivial one-parameter subgroups  $\lambda$ of $T$.
    \item $X = \{f=0\}$ is \emph{stable} if and only if $\mu(f, \lambda) > 0$, i.e. if $\min_{x^I\in \operatorname{Supp}(f)}\{\langle I, \lambda \rangle\}< 0$, for all non-trivial one-parameter subgroups  $\lambda$ of $T$.
\end{enumerate}
\end{lemma}

\section{Stability of \texorpdfstring{$k$}{TEXT}-tuples of Hypersurfaces and Log Canonical Thresholds}\label{thresholds section}
In this Section we will study how we can use the Hilbert-Mumford numerical criterion, to study a connection between GIT and log canonical thresholds. We will also, study the above through the scope of VGIT. It serves as a generalisation of the results presented in Zanardini \cite{zanardini2021stability}, where the link between log canonical thresholds and the GIT stability of pencils of hypersurfaces are studied, with an emphasis on the pencils of curves. Here, we generalise this setting, considering complete intersections of  $k$ hypersurfaces of degree $d$ in $\mathbb{P}^n$, and hence, elements $\mathcal{T}\in \operatorname{Gr}\Big(k,{\binom{n+d}{d}}\Big) = \mathcal{R}_{n,d,k}$. 

\subsection{GIT and log Canonical Thresholds}\label{git_lc_sect}

As before, we fix a maximal torus $T$ of $G$, and a coordinate system such that $T$ is diagonal.
As in \cite[\S 4.1]{me}, the Hilbert-Mumford function is given by
\begin{equation*}
        \mu(f_1\wedge\dots \wedge f_k, \lambda)\vcentcolon= -\min \left\{\sum_{i=1}^k\langle I_i,\lambda\rangle  \text{ }\Big| (I_1,\dots,I_k)\in (\Xi_d)^k, I_i \neq I_j  \text{ if } i\neq j \text{ and }x^{I_i}\in \operatorname{Supp}(f_i)\right\},
\end{equation*}
where $\lambda$ is a normalised one-parameter subgroup.

Let $I_i = (d_{i,0}, \dots, d_{i,n})$ with $i = 1, \dots, k $, $\sum_{j=0}^nd_{i,j} = d$ be distinct monomials with $I_i \in \operatorname{Supp}(f_i)$. For any normalised one-parameter subgroup $\lambda(s) = \operatorname{Diag}(s^{a_0}, \dots, s^{a_n})$ and since $d_{i,n} = d - \sum_{j=0}^{n-1}d_{i,j}$ and $a_n  = -\sum_{k=0}^{n-1}a_k $ we have:

\begin{equation*}
    \begin{split}
        \mu(f_1\wedge\dots \wedge f_k, \lambda) &= -\min_{x^{I_i}\in\operatorname{Supp(f_i)}} \bigg\{\sum_{i=1}^k\langle I_i, \lambda \rangle \bigg| (1) \bigg\}\\
        & = -\min_{x^{I_i}\in\operatorname{Supp}(f_i)} \bigg\{ \sum_{i=1}^k \bigg(\sum_{j=0}^nd_{i,j}a_{j}\bigg)\bigg| (1) \bigg\}\\
        &= -\min_{x^{I_i}\in\operatorname{Supp}(f_i)} \bigg\{ \sum_{j=0}^n \bigg(\sum_{i=1}^k d_{i,j}a_{j}\bigg)\bigg| (1)\bigg\}\\
        &=-\min_{x^{I_i}\in\operatorname{Supp}(f_i)} \bigg\{ \sum_{j=0}^{n-1} \bigg(\sum_{i=1}^k d_{i,j}a_{j}\bigg) -\bigg(kd- \sum_{i=1}^k\sum_{j=0}^{n-1}d_{i,j}\bigg)a_n\bigg| (1)\bigg\}\\
        &=-\min_{x^{I_i}\in\operatorname{Supp}(f_i)} \bigg\{ \sum_{j=0}^{n-1} \bigg(\sum_{i=1}^k d_{i,j}(a_{j}-a_n)\bigg) -kda_n\bigg| (1)\bigg\}\\
        &=-\min_{x^{I_i}\in\operatorname{Supp}(f_i)} \bigg\{ \sum_{j=0}^{n-1} \bigg(\sum_{i=1}^k d_{i,j}(a_{j}-a_n)\bigg| (1)\bigg)\bigg\} +\frac{kd}{n+1}\bigg(\sum_{k=0}^{n-1}a_k - na_n\bigg)
    \end{split}
\end{equation*}
where condition $(1)$ refers to the condition $(I_1,\dots,I_k)\in (\Xi_d)^k, I_i \neq I_j  \text{ if } i\neq j \text{ and }x^{I_i}\in \operatorname{Supp}(f_i)$. 
Throughout, we will denote by $\mathcal{T}$ an element of the Grassmanian $\operatorname{Gr}(k,W)$. Recall, that $\mathcal{R}$ parametrises the space of tuples $\{\sum_{i=1}^kz_if_i|(z_1\colon\dots\colon z_k)\in \mathbb{P}^{k-1}\}$ of $k$ hypersurfaces, and as such, we can write $\mathcal{T} = \{\sum_{i=1}^kz_if_i|(z_1\colon\dots\colon z_k)\in \mathbb{P}^{k-1}\}$. We define the following.
\begin{definition}\label{omega for ktuples}
Fix a maximal torus $T$. For any $k$-tuple $\mathcal{T}\in \mathcal{R}_{n,d,k}$ and normalised one-parameter subgroup $\lambda$ we define the \emph{affine weight} of $\mathcal{T}$ as  

$$\omega(\mathcal{T},\lambda)\vcentcolon = \min_{x^{I_i}\in\operatorname{Supp(f_i)}} \bigg\{ \sum_{j=0}^{n-1} \bigg(\sum_{i=1}^k d_{i,j}(a_{j}-a_n)\bigg)\bigg\}.$$
\end{definition}

With this definition and the above discussion, we can reformulate the Hilbert-Mumford numerical criterion:

\begin{lemma}\label{hilbert-mumford-interms of omega}
With respect to a maximal torus $T$, $\mathcal{T}$ is unstable (respectively, non-stable) if for some normalised one-parameter subgroup $\lambda$
$$\omega(\mathcal{T},\lambda) > \frac{kd}{n+1}\bigg(\sum_{k=0}^{n-1}a_k - na_n\bigg) \text{ (resp. } \geq).$$
Similarly, $\mathcal{T}$ is (semi-)stable if for all normalized one-parameter subgroups $\lambda$

$$\omega(\mathcal{T},\lambda) < \frac{kd}{n+1}\bigg(\sum_{k=0}^{n-1}a_k - na_n\bigg) \text{ (resp. } \leq).$$
\end{lemma}

We can also define the affine weight of a hypersurface $f$, by following the same discussion for the case $k=1$.

\begin{definition}[{\cite[Definition 3.3]{zanardini2021stability}}]\label{affine weight}

Fix a maximal torus $T$. For a hypersurface $f$ of degree $d$ and normalised one-parameter subgroup $\lambda$ we define its \emph{affine weight} as 
$$\omega(f,\lambda) \vcentcolon= \min\bigg\{ \sum_{i=0}^{n-1}d_i(a_j-a_n)\Big| I = (d_1, \dots, d_n)\in \operatorname{Supp}(f_i) \bigg\}.$$
\end{definition}

This definition allows us to rewrite Lemma \ref{HM criterion for hypersurfaces} as follows.

\begin{proposition}\label{semistab for hypersurfaces}
With respect to a maximal torus $T$, $f$ is unstable (respectively, non-stable) if for some normalised one-parameter subgroup $\lambda$
$$\omega(f,\lambda) > \frac{d}{n+1}\bigg(\sum_{k=0}^{n-1}a_k - na_n\bigg) \text{ (resp. } \geq).$$

\end{proposition}

\begin{proposition}[Analogue of {\cite[Proposition 3.5]{zanardini2021stability}}]\label{contr crit}

Given a $k$-tuple $\mathcal{T}\in \mathcal{R}_{n,d,k}$ and $k$ distinct hypersurfaces $g_1, \dots, g_k \in \mathcal{T}$ we have

$$\omega(g_i,\lambda)\leq \sum_{i=1}^k \omega(g_i,\lambda)\leq \omega(\mathcal{T},\lambda)$$
for all one-parameter subgroups $\lambda$.
\end{proposition}
\begin{proof}
Notice that $\omega(f,\lambda)\geq0$ for all hypersurfaces $f$ and all normalised one-parameter subgroups $\lambda$, since $a_i\geq a_n$ for all $i<n$. This gives the left-hand side of the inequality. For the right-hand side of the inequality notice that 
\begin{equation*}
    \begin{split}
       \omega(\mathcal{T},\lambda)&=  \sum_{j=0}^{n-1} \bigg(\sum_{i=1}^k d_{i,j}(a_{j}-a_n)\bigg)\\
       &= \sum_{j=0}^{n-1} \bigg( d_{1,j}(a_{j}-a_n)\bigg) + \sum_{j=0}^{n-1} \bigg( d_{2,j}(a_{j}-a_n)\bigg) +\dots+ \sum_{j=0}^{n-1} \bigg( d_{k,j}(a_{j}-a_n)\bigg)\\
       &\geq \omega(g_1,\lambda)+\dots+ \omega(g_k,\lambda)
    \end{split}
\end{equation*}
since the $g_i$ are distinct.
\end{proof}



\begin{proposition}\label{equality of omegas}
Let $F =\{f=0\} \in \mathcal{T}$ be a hypersurface of degree $d$, and $\lambda$ a normalised one-parameter subgroup. Then there exist $k-1$ hypersurfaces $g_1, \dots, g_{k-1}\in \mathcal{T}$ such that 
$$\omega(\mathcal{T},\lambda) = \omega(f,\lambda)+\sum_{i=1}^{k-1} \omega(g_i,\lambda).$$
\end{proposition}
\begin{proof}
Let $\mathcal{T}$ be generated by $k$ hypersurfaces of degree $d$, $f_i$. Since $F =\{f=0\} \in \mathcal{T}$, $f$ is a linear combination of the $f_i$ that generate $\mathcal{T}$, i.e. $f = \sum_{i=1}^k \xi_if_i$ for some $(\xi_1:\dots:\xi_k)\in \mathbb{P}^{k-1}$. Then, for a fixed normalised one-parameter subgroup $\lambda$, with respect to the $\lambda$-order introduced before, the minimum monomial $I= \min\{\operatorname{Supp}(f)\}=\min\{\operatorname{Supp}(f_l)\}$, where without loss of generality we can assume that $f_l = f_1$ (if not we can rearrange the generators such that $f_l = f_1$). Hence, let $I = (d_{1,0}, \dots, d_{1,n})$
Then, by taking $g_i \in \mathcal{T}$ recursively such that 
$$\min\{\operatorname{Supp}(g_i)\} = \min\bigg\{\operatorname{Supp}(f_{i+1}) \setminus\big(\bigcup_{k\leq i} \operatorname{Supp}(f_k) \big) \bigg\}$$
the result follows, as we obtain $k-1$ distinct monomials $I_i = (d_{i,0},\dots, d_{i,n})$, $2\leq i\leq k$ such that 
$$\omega(\mathcal{T},\lambda)\vcentcolon =  \sum_{j=0}^{n-1} \bigg(\sum_{i=1}^k d_{i,j}(a_{j}-a_n)\bigg) = \omega(f,\lambda)+\sum_{i=1}^{k-1}\omega(g_i,\lambda).$$
\end{proof}
\begin{corollary}\label{max cond}
Let $F = \{f = 0\} \in \mathcal{T}$ be a hypersurface of degree d, and $\lambda$ a one-parameter subgroup. Then there exist $k-1$ hypersurfaces $g_1, \dots, g_{k-1}\in \mathcal{T}$ such that 
$$\omega(\mathcal{T},\lambda) \leq k \max\bigg\{\omega(f,\lambda),\omega(g_1,\lambda), \dots, \omega(g_{k-1},\lambda)\bigg\}.$$
\end{corollary}

\begin{corollary}\label{all semi-stable}
If a $k$-tuple $\mathcal{T}\in \mathcal{R}_{n,d,k}$ has only semistable (respectively, stable) members, then $\mathcal{T}$ is semistable (respectively stable).
\end{corollary}
\begin{proof}
Let $f,g_1,\dots, g_{k-1}\in \mathcal{T}$ be $k$ semistable hypersurfaces as in Corollary \ref{max cond}. Then by Proposition \ref{semistab for hypersurfaces} for all $\lambda$ and $i$
$$\frac{\omega(f,\lambda)}{\bigg(\sum_{k=0}^{n-1}a_k - na_n\bigg)} \leq \frac{d}{n+1},$$
$$\frac{\omega(g_i,\lambda)}{\bigg(\sum_{k=0}^{n-1}a_k - na_n\bigg)} \leq \frac{d}{n+1},$$
and by Corollary \ref{max cond}
\begin{equation*}
\begin{split}
    \frac{\omega(\mathcal{T},\lambda)}{\bigg(\sum_{k=0}^{n-1}a_k - na_n\bigg)} &\leq k \max\bigg\{\omega(f,\lambda),\omega(g_1,\lambda), \dots, \omega(g_{k-1},\lambda)\bigg\}\\
    &\leq\frac{kd}{n+1}.
\end{split}
\end{equation*}
\end{proof}
\begin{corollary}
    If $\mathcal{T}$ contains only smooth members then $\mathcal{T}$ is stable.
\end{corollary}

The results of \cite{zanardini2021stability} can also be extended as follows:

\begin{theorem}[Analogue of {\cite[Theorem 3.14]{zanardini2021stability}}]

If $\mathcal{T}\in \mathcal{R}_{n,d,k}$ contains at worst one strictly semistable hypersurface (and all
other hypersurfaces in $\mathcal{T}$ are stable), then $\mathcal{T}$ is stable.
\end{theorem}
\begin{proof}
Let $f\in \mathcal{T}$ be strictly semistable, i.e. by Proposition \ref{semistab for hypersurfaces} since $f$ is strictly semi-stable, for some normalised $\lambda$, we have
$$\omega(f,\lambda) = \frac{d}{n+1}\bigg(\sum_{k=0}^{n-1}a_k - na_n\bigg).$$
By Proposition \ref{equality of omegas} there exist $g_1,\dots,g_{k-1}\in \mathcal{T}$ which are stable and 
\begin{equation*}
    \begin{split}
        \frac{\omega(\mathcal{T},\lambda)}{\bigg(\sum_{k=0}^{n-1}a_k - na_n\bigg)}&= \frac{\omega(f,\lambda)+\sum_{i=1}^{k-1} \omega(g_i,\lambda)}{\bigg(\sum_{k=0}^{n-1}a_k - na_n\bigg)}\\
    &\leq \frac{d}{n+1}+\sum_{i=1}^{k-1}\frac{d}{n+1}\\
    &=\frac{kd}{n+1}.
    \end{split}
\end{equation*}
\end{proof}

\begin{theorem} [Analogue of {\cite[Proposition 4.6]{zanardini2021stability}}]\label{zan_lct_analog}
For a $k$-tuple $\mathcal{T}\in \mathcal{R}_{n,d,k}$, with $k>1$, and any base point $p$ of $\mathcal{T}$, there exists a one-parameter subgroup $\lambda$ such that for any hypersurface $F =(f=0)$ in $\mathcal{T}$ 

$$\frac{\sum_{i=0}^{n-1}(a_i)-na_n}{\omega(\mathcal{T}, \lambda)}\leq \mathrm{lct}_p(\mathbb{P}^n, F).$$

\end{theorem}
\begin{proof}
The proof follows the idea of proof of  \cite[ Proposition 4.1]{zanardini2021stability}. Without loss of generality, we choose coordinates $(x_0:\dots:x_n)$ in $\mathbb{P}^n$ such that $p = (0:\dots:0:1)$. Let $\lambda$ be a normalized one-parameter subgroup with 
$$\lambda = \operatorname{Diag}(s,s^{a_1}, \dots, s^{a_{n-1}}, s^{-1-\sum_{i=1}^{n-1}a_{i}}).$$

Notice then that 

\begin{equation*}
    \begin{split}
        \sum_{i=0}^{n-1}(a_i-a_n)& = 2+\sum_{i=0}^{n-1}a_i+\sum_{j=0}^{n-1}\big(1+\sum_{i=0}^{n-1}a_i+a_j\big)\\
        & = (n+1)+(n+1)\sum_{i=0}^{n-1}a_i\\
        &=(n+1)\big(1+\sum_{i=0}^{n-1}a_i\big)
    \end{split}
\end{equation*}
and

\begin{equation*}
    \begin{split}
        \omega(\mathcal{T}, \lambda)&= \sum_{j=0}^{n-1} \bigg(\sum_{i=1}^k d_{i,j}(a_{j}-a_n)\bigg)\\
        &\geq\sum_{i=1}^{2} \bigg(\sum_{j=0}^{n-1} d_{i,j}(a_{j}-a_n)\bigg)\\
        &\geq \bigg[ (n+1)+(n+1)\sum_{i=0}^{n-1}a_i\bigg]
    \end{split}
\end{equation*}
where the last inequality follows by \cite[Proposition 4.1]{zanardini2021stability}. Hence,
$$\frac{\sum_{i=0}^{n-1}(a_i-a_n)}{\omega(\mathcal{T}, \lambda)} =\frac{(n+1)\big(1+\sum_{i=0}^{n-1}a_i\big)}{\sum_{j=0}^{n-1} \bigg(\sum_{i=1}^k d_{i,j}(1+a_{j}+\sum_{i=0}^{n-1}a_i)\bigg)} \leq 1.$$

For contradiction, assume that there exists $F =\{f=0\} \in \mathcal{T}$ such that
$$\mathrm{lct}_p(\mathbb{P}^n, F)<\frac{\sum_{i=0}^{n-1}(a_i-a_n)}{\omega(\mathcal{T}, \lambda)}.$$
Then, let $\tilde{F}(u_1,\dots,u_n) = f(x_0:\dots:x_{n-1}:1)$, where $u_i = \frac{x_{i-1}}{x_n}$ defined in a neighbourhood around $p$, which is enough to compute $\mathrm{lct}_p$ since it is a local invariant. Also, assign weights $\omega(u_1) = a_0 - a_n$, $\omega(u_i) = a_{i-1} - a_n$, $\omega(u_n) = a_{n-1} - a_n$. Consider the finite morphism $\phi \colon \mathbb{C}^n\rightarrow \mathbb{C}^n$ where $(u_1,\dots,u_n)\mapsto (u_1^{\omega(u_1)},\dots,u_n^{\omega(u_n)})$. Then, let $H_{u_i} = \{u_i =0\}$ and

$$\Delta\vcentcolon= \sum_{i=1}^n(1-\omega(u_i))H_{u_i} +c\tilde{F}(u_1^{\omega(u_1)},\dots,u_n^{\omega(u_n)})$$
for some $c\in\mathbb{Q}\cap [0,1]$. Then,

$$\phi^*(K_{\mathbb{C}^n}+c\tilde{F}(u_1,\dots,u_n) ) = K_{\mathbb{C}^n}+\Delta.$$
We also know that the pair $(\mathbb{C}^n,\tilde{F})$ is log canonical
at the origin if and only if the pair $(\mathbb{C}^n,\Delta)$ is log canonical at the origin. Let 
$$c = \frac{\sum_{i=1}^n\omega(u_i)}{\omega(\mathcal{T},\lambda)}$$
where $c> \mathrm{lct}_p(\mathbb{P}^n, F) = \mathrm{lct}_{\vec{0}}(\mathbb{C}^n, \tilde{F})$ by the assumption. Blowing up $\mathbb{C}^n$ at the origin, we then have that for the log discrepancy of $\Delta$ with respect to the exceptional divisor $E$ of the blow up
$$a(E;\mathbb{C}^n,\Delta) = -1-\sum_{i=1}^n\omega(u_i)-c\omega(f,\lambda)<-1$$
which in turn would imply that $\omega(f,\lambda)>\omega(\mathcal{T},\lambda)$ which contradicts Proposition \ref{contr crit}.

\end{proof}

\begin{corollary}\label{if_ss_lct_condition}
If the $k$-tuple $\mathcal{T}\in \mathcal{R}_{n,d,k}$, with $k>1$ is (semi-)stable, then
for any hypersurface $F =\{f=0\} \in\mathcal{T}$ and any base point $p$ of $\mathcal{T}$ 
$$\frac{n+1}{kd}< \mathrm{lct}_p(\mathbb{P}^n, F) \text{ (respectively } \leq).$$

\end{corollary}
\begin{proof}
Since $\mathcal{T}$ is (semi-)stable, for all normalized one-parameter subgroups $\lambda$ we have
$$\frac{\sum_{i=0}^{n-1}(a_i-a_n)}{\omega(\mathcal{T}, \lambda)}\geq \frac{n+1}{kd}$$
by Lemma \ref{hilbert-mumford-interms of omega}.
\end{proof}

\begin{proposition}\label{rel wrt lct}
Given a tuple $\mathcal{T}\in \mathcal{R}_{n,d,k}$ we have that for any one-parameter subgroup $\lambda$ there exists $F =\{f=0\} \in \mathcal{T}$ such that

$$\frac{\omega(\mathcal{T},\lambda)}{\bigg(\sum_{k=0}^{n-1}a_k - na_n\bigg)}\leq \frac{k}{\mathrm{lct}(\mathbb{P}^n,F)}.$$

\end{proposition}
\begin{proof}
By \cite[Proposition 8.13]{kollar} we have that 
$$\frac{\omega(f,\lambda)}{\bigg(\sum_{k=0}^{n-1}a_k - na_n\bigg)} \leq \frac{1}{\mathrm{lct}(\mathbb{P}^n,F)},$$
and we obtain the result by Corollary \ref{max cond}.

\end{proof}
\begin{corollary}\label{zanardini corollary to go in intro}
If a tuple $\mathcal{T}\in \mathcal{R}_{n,d,k}$ is such that $\mathrm{lct}(\mathbb{P}^n,F)\geq \frac{n+1}{d}$ (respectively, $>\frac{n+1}{d}$)
for any hypersurface $F = \{f=0\}$ in $\mathcal{T}$, then $\mathcal{T}$ is semistable (respectively, stable).
\end{corollary}
\begin{proof}
Let $f\in \mathcal{T}$ a hypersurface with $\mathrm{lct}(\mathbb{P}^n,F)\geq \frac{n+1}{d}$ (respectively $>\frac{n+1}{d}$). Then, for any normalized one-parameter subgroup $\lambda$,

\begin{equation*}
    \begin{split}
        \frac{\omega(\mathcal{T},\lambda)}{\bigg(\sum_{k=0}^{n-1}a_k - na_n\bigg)}&\leq \frac{k}{\mathrm{lct}(\mathbb{P}^n,F)}\\
        &\leq \frac{kd}{n+1}\quad (< \frac{kd}{n+1}\text{ resp.).}
    \end{split}
\end{equation*}
\end{proof}

\begin{corollary}
If $\mathcal{T}$ is (semi-)stable then 
$$\frac{n+1}{kd}\leq \mathrm{lct}(\mathbb{P}^n, \mathcal{T})\quad \text{respectively, }<$$
\end{corollary}
\begin{proof}
    This is a direct consequence of \cite[Theorem 4.8, Lemma 8.6]{kollar}, Bertini's Theorem and Proposition \ref{rel wrt lct}.
\end{proof}

\subsection{VGIT and log Canonical Thresholds}\label{VGIT section}
Consider now the GIT quotient $\mathcal{R}_m\sslash G$
which parametrizes tuples $(\mathcal{T},H_1,\dots,H_m)$ of complete intersections and $m$ hyperplanes in $\mathbb{P}^n$. 


Recall that from \cite[Lemma 4.10, \S 4.1]{me}, with respect to an ample linearisation $\mathcal{L} = \mathcal{O}(a,\vec{b})$ the Hilbert-Mumford function is
\begin{equation*}
    \begin{split}
        \mu_{\vec{t}}(S, H_1,\dots, H_m, \lambda) &\vcentcolon= \mu_{\vec{t}}(f_1\wedge\dots \wedge f_k, h_1,\dots, h_m, \lambda) \\
        &\coloneqq  \mu(f_1\wedge\dots \wedge f_k, \lambda) + \sum_{p=1}^mt_p\mu(H_p, \lambda)      \\
        &= \max \left\{\sum_{i=1}^k\langle I_i, \lambda\rangle \Big| (I_1,\dots,I_k)\in (\Xi_d)^k \text{, } I_i \neq I_j  \text{ if } i\neq j \text{, }x^{I_i}\in \operatorname{Supp}(f_i) \right\}\\
        &+ \sum_{p=1}^mt_p\max\{r_j| h_{j,p}\neq0, h_{j,p} \in \operatorname{Supp}(h_p)\},
    \end{split}
\end{equation*}
where $\vec{t} = (t_1,\dots,t_m)$ with $t_i= \frac{b_i}{a} \in \mathbb{Q}_{>0}$. Combining \cite[Definition 4.12]{me} and Lemma \ref{hilbert-mumford-interms of omega} we obtain:

\begin{lemma}\label{vgit hilbert}
With respect to a maximal torus $T$, the tuple $(\mathcal{T},H_1,\dots,H_m)$ is $\vec{t}-$unstable (respectively, $\vec{t}-$non-stable) if for some $\lambda$
$$\omega(\mathcal{T},\lambda)+\sum_{i=1}^mt_i\omega(H_i,\lambda) > \frac{kd+\sum_{i=1}^mt_i}{n+1}\bigg(\sum_{k=0}^{n-1}a_k - na_n\bigg) \text{ (resp. } \geq).$$
\end{lemma}
\begin{proof}
From \cite[Lemma 4.12]{me}, we have that $\mu_{\vec{t}}(\mathcal{T},H_1,\dots,H_m,\lambda) = \mu(\mathcal{T},\lambda) +\sum_{i=0}^mt_i\mu(H_i,\lambda)$. The result then follows.
\end{proof}

We expand Proposition \ref{equality of omegas} as follows:

\begin{proposition}\label{equality of omegas-vgit}
Let $F =\{f=0\} \in \mathcal{T}$ be a hypersurface of degree d, and $\lambda$ a one-parameter subgroup. Then there exist $k-1$ hypersurfaces $g_1, \dots, g_{k-1}\in \mathcal{T}$ such that 
$$\omega(\mathcal{T},\lambda)+\sum_{i=1}^mt_i\omega(H_i,\lambda) = \omega(f,\lambda)+\sum_{i=1}^{k-1} \omega(g_i,\lambda)+\sum_{i=1}^mt_i\omega(H_i,\lambda).$$
\end{proposition}

By Fujita \cite{fujita2017kstability}, the log-canonical threshold  $\mathrm{lct}(\mathbb{P}^n, H_i) = 1$ for all hyperplanes $H_i$, and hence 
$$\frac{\omega(H_i,\lambda)}{\bigg(\sum_{k=0}^{n-1}a_k - na_n\bigg)} \leq 1,$$
so in combination with Proposition \ref{rel wrt lct} we obtain the following. 

\begin{proposition}\label{rel wrt lct vgit}
Given a tuple $(\mathcal{T},H_1\dots,H_m)\in \mathcal{R}_{m}$ we have that for any one-parameter subgroup $\lambda$ there exists $F =\{f=0\} \in \mathcal{T}$ such that

$$\frac{\omega(\mathcal{T},\lambda)+\sum_{i=1}^mt_i\omega(H_i,\lambda)}{\bigg(\sum_{k=0}^{n-1}a_k - na_n\bigg)}\leq \frac{k}{\mathrm{lct}(\mathbb{P}^n,F)}+\sum_{i=1}^mt_i.$$

\end{proposition}
\begin{corollary}\label{vgit_corollary_zan}
If a tuple $(\mathcal{T},H_1\dots,H_m)\in \mathcal{R}_{m}$ is such that $\mathrm{lct}(\mathbb{P}^n,F)\geq \frac{k(n+1)}{kd-n\sum_{i=1}^mt_i}$ (respectively, $>\frac{k(n+1)}{kd-n\sum_{i=1}^mt_i}$)
for any hypersurface $f$ in $\mathcal{T}$, then $(\mathcal{T},H_1\dots,H_m)\in \mathcal{R}_{m}$ is $\vec{t}-$semistable (respectively, $\vec{t}-$stable).
\end{corollary}
\begin{proof}
From Proposition \ref{zan_lct_analog}, let $f\in \mathcal{T}$ be such that $\mathrm{lct}(\mathbb{P}^n,F)\geq \frac{k(n+1)}{kd-n\sum_{i=1}^mt_i}$ (respectively, $>\frac{k(n+1)}{kd-n\sum_{i=1}^mt_i}$). Then for all $\lambda$,
\begin{equation*}
    \begin{split}
        \frac{\omega(\mathcal{T},\lambda)+\sum_{i=1}^mt_i\omega(H_i,\lambda)}{\bigg(\sum_{k=0}^{n-1}a_k - na_n\bigg)}&\leq \frac{k}{\mathrm{lct}(\mathbb{P}^n,F)}+\sum_{i=1}^mt_i\\
        &\leq  \frac{kd-n\sum_{i=1}^mt_i}{n+1}+\sum_{i=1}^mt_i \text{ (respectively, } <)\\
        &\leq  \frac{kd+\sum_{i=1}^mt_i}{n+1} \text{ (respectively, } <).
    \end{split}
\end{equation*}

\end{proof}

\printbibliography

@article{gallardo_martinez-garcia_2018,
title={Variations of Geometric Invariant Quotients for Pairs, a Computational Approach}, 
volume={146},
DOI={10.1090/proc/13950}, number={6}, 
journal={Proceedings of the American Mathematical Society},
author={Gallardo, Patricio and Martinez-Garcia, Jesus},
year={2018}, 
pages={2395–2408}
}

@article{odaka_spotti_sun_2016, 
title={Compact moduli spaces of {D}el {P}ezzo surfaces and {K}\"{a}hler–{E}instein metrics}, 
volume={102}, 
DOI={10.4310/jdg/1452002879},
number={1}, 
journal={Journal of Differential Geometry}, 
author={Odaka, Yuji and Spotti, Cristiano and Sun, Song},
year={2016},
pages={127–172}
}

@article{thaddeus_1996, 
title={Geometric invariant theory and flips}, 
volume={6}, 
DOI={https://doi.org/10.1090/S0894-0347-96-00204-4},
journal={J. Amer. Math. Soc. }, 
author={Thaddeus, Michael},
year={1996},
pages={691-723}
}

@article {dolgachev_1994,
    AUTHOR = {Dolgachev, Igor V. and Hu, Yi},
     TITLE = {Variation of geometric invariant theory quotients},
      NOTE = {With an appendix by Nicolas Ressayre},
   JOURNAL = {Inst. Hautes \'{E}tudes Sci. Publ. Math.},
  FJOURNAL = {Institut des Hautes \'{E}tudes Scientifiques. Publications
              Math\'{e}matiques},
    NUMBER = {87},
      YEAR = {1998},
     PAGES = {5--56},
      ISSN = {0073-8301},
   MRCLASS = {14L24},
  MRNUMBER = {1659282},
MRREVIEWER = {P. E. Newstead},
       URL = {http://www.numdam.org/item?id=PMIHES_1998__87__5_0},
}

@article {Gallardo_2020,
    AUTHOR = {Gallardo, Patricio and Martinez-Garcia, Jesus and Spotti,
              Cristiano},
     TITLE = {Applications of the moduli continuity method to log {K}-stable
              pairs},
   JOURNAL = {J. Lond. Math. Soc. (2)},
  FJOURNAL = {Journal of the London Mathematical Society. Second Series},
    VOLUME = {103},
      YEAR = {2021},
    NUMBER = {2},
     PAGES = {729--759},
      ISSN = {0024-6107},
   MRCLASS = {32Q20 (14D22 14J10 14J45 14L24)},
  MRNUMBER = {4230917},
MRREVIEWER = {P. E. Newstead},
       DOI = {10.1112/jlms.12390},
       URL = {https://0-doi-org.serlib0.essex.ac.uk/10.1112/jlms.12390},
}

@article {zanardini2021stability,
    AUTHOR = {Zanardini, Aline},
     TITLE = {A note on the stability of pencils of plane curves},
   JOURNAL = {Math. Z.},
  FJOURNAL = {Mathematische Zeitschrift},
    VOLUME = {300},
      YEAR = {2022},
    NUMBER = {2},
     PAGES = {1741--1751},
      ISSN = {0025-5874},
   MRCLASS = {14L24 (14H10)},
  MRNUMBER = {4363795},
MRREVIEWER = {Claudia R. Alc\'{a}ntara},
       DOI = {10.1007/s00209-021-02857-w},
       URL = {https://0-doi-org.serlib0.essex.ac.uk/10.1007/s00209-021-02857-w},
}

@article {zhang_2018,
    AUTHOR = {Gallardo, Patricio and Martinez-Garcia, Jesus and Zhang,
              Zheng},
     TITLE = {Compactifications of the moduli space of plane quartics and
              two lines},
   JOURNAL = {Eur. J. Math.},
  FJOURNAL = {European Journal of Mathematics},
    VOLUME = {4},
      YEAR = {2018},
    NUMBER = {3},
     PAGES = {1000--1034},
      ISSN = {2199-675X},
   MRCLASS = {14L24 (14J17 14J28 14Q10 32G20)},
  MRNUMBER = {3851127},
MRREVIEWER = {Ronan Terpereau},
       DOI = {10.1007/s40879-018-0248-7},
       URL = {https://0-doi-org.serlib0.essex.ac.uk/10.1007/s40879-018-0248-7},
}

@misc{liu2020kstability,
      title={K-stability of cubic fourfolds}, 
      author={Yuchen Liu},
      year={2020},
      eprint={2007.14320},
      archivePrefix={arXiv},
      primaryClass={math.AG}
}

@article {fujita2017kstability,
    AUTHOR = {Fujita, Kento},
     TITLE = {K-stability of log {F}ano hyperplane arrangements},
   JOURNAL = {J. Algebraic Geom.},
  FJOURNAL = {Journal of Algebraic Geometry},
    VOLUME = {30},
      YEAR = {2021},
    NUMBER = {4},
     PAGES = {603--630},
      ISSN = {1056-3911},
   MRCLASS = {14E30 (14N20 32Q26)},
  MRNUMBER = {4372401},
       DOI = {10.1090/jag/783},
       URL = {https://0-doi-org.serlib0.essex.ac.uk/10.1090/jag/783},
}

@article {kollar,
    AUTHOR = {Koll\'{a}r, J.},
     TITLE = {Singularities of pairs},
   JOURNAL = { Algebraic geometry—Santa Cruz},
      YEAR = {1995, 1997},
     PAGES = {221–-287},
       URL = {https://web.math.princeton.edu/~kollar/FromMyHomePage/janosbib2017.pdf},
}

@article {liu-xu-threefolds,
    AUTHOR = {Liu, Yuchen and Xu, Chenyang},
     TITLE = {K-stability of cubic threefolds},
   JOURNAL = {Duke Math. J.},
  FJOURNAL = {Duke Mathematical Journal},
    VOLUME = {168},
      YEAR = {2019},
    NUMBER = {11},
     PAGES = {2029--2073},
      ISSN = {0012-7094},
   MRCLASS = {14L24 (14E30 14J30 32Q20)},
  MRNUMBER = {3992032},
MRREVIEWER = {Yuji Odaka},
       DOI = {10.1215/00127094-2019-0006},
       URL = {https://0-doi-org.serlib0.essex.ac.uk/10.1215/00127094-2019-0006},
}

@article {laza-cubics,
    AUTHOR = {Laza, Radu},
     TITLE = {The moduli space of cubic fourfolds},
   JOURNAL = {J. Algebraic Geom.},
  FJOURNAL = {Journal of Algebraic Geometry},
    VOLUME = {18},
      YEAR = {2009},
    NUMBER = {3},
     PAGES = {511--545},
      ISSN = {1056-3911},
   MRCLASS = {14J10 (14J17 14J35 14L24)},
  MRNUMBER = {2496456},
MRREVIEWER = {Sean Lawton},
       DOI = {10.1090/S1056-3911-08-00506-7},
       URL = {https://0-doi-org.serlib0.essex.ac.uk/10.1090/S1056-3911-08-00506-7},
}

@misc{me,
  doi = {10.48550/ARXIV.2212.09332},
  
  url = {https://arxiv.org/abs/2212.09332},
  
  author = {Papazachariou, Theodoros Stylianos},
  
  title = {K-moduli of log Fano complete intersections},
  
  publisher = {arXiv},
  eprint={2212.09332},
  archivePrefix={arXiv},
  primaryClass={math.AG},
  year = {2022},
  
  copyright = {arXiv.org perpetual, non-exclusive license}
}

@misc{zanardini-hattori,
  doi = {10.48550/ARXIV.2212.09364},
  
  url = {https://arxiv.org/abs/2212.09364},
  
  author = {Hattori, Masafumi and Zanardini, Aline},
  
  title = {On the GIT stability of linear systems of hypersurfaces in projective space},
  
  publisher = {arXiv},
  eprint={2212.09364},
  archivePrefix={arXiv},
  primaryClass={math.AG},
  year = {2022},
  
  copyright = {arXiv.org perpetual, non-exclusive license}
}

@article{miranda_1980, title={On the stability of pencils of cubic curves}, volume={102}, DOI={10.2307/2374184}, number={6}, journal={American Journal of Mathematics}, author={Miranda, Rick}, year={1980}, pages={1177}}

@article{wall_1983, title={Geometric invariant theory of linear systems}, volume={93}, DOI={10.1017/s0305004100060321}, number={1}, journal={Mathematical Proceedings of the Cambridge Philosophical Society}, author={Wall, C. T.}, year={1983}, pages={57–62}}

\end{document}